\documentclass[10pt]{amsart}

\usepackage{amsfonts}
\usepackage{amsthm}
\usepackage{amssymb}
\usepackage{amsmath}

\newcommand{\alphabet}{\mathfrak{n}}
\newcommand{\shiftspace}{\alphabet^{\mathbb{Z}}}
\newcommand{\orbit}{\text{Orb}}
\newcommand{\shiftorbit}{\overline{\orbit}}
\newcommand{\blanksymbol}{\square}
\newcommand{\lcmprod}{\text{lcm}}
\newcommand{\natnump}{\mathbb{N}^+}
\newcommand{\natnum}{\mathbb{N}}

\newcommand{\spaceoftoeplitzsubshift}{\mathcal{T}_{\alphabet}}
\newcommand{\spaceoftoeplitzsubshiftsep}{\spaceoftoeplitzsubshift^{*}}
\newcommand{\spaceoftoeplitzsubshiftgr}{\spaceoftoeplitzsubshift^{**}}

\newcommand{\seqshift}{\sigma}
\newcommand{\odometer}{\mathtt{Odo}}
\newcommand{\odometeraddition}{\lambda}
\newcommand{\fin}{\text{fin}}
\newcommand{\Fin}{\text{Fin}}

\newtheorem{theorem}{Theorem}

\newtheorem{lemma}[theorem]{Lemma}
\newtheorem{corollary}[theorem]{Corollary}
\newtheorem{proposition}[theorem]{Proposition}

\newtheorem{question}{Question}

\title{The complexity of the topological conjugacy problem for Toeplitz subshifts}
\author{Burak Kaya}
\date{\today}
\address{
Department of Mathematics, Rutgers University\\
110 Frelinghuysen Road, Hill Center,
Piscataway, NJ 08854, USA\\
}
\email{bkaya@math.rutgers.edu}
\begin{document}

\begin{abstract} In this paper, we analyze the Borel complexity of the topological conjugacy relation on Toeplitz subshifts. More specifically, we prove that topological conjugacy of Toeplitz subshifts with separated holes is hyperfinite. Indeed, we show that the topological conjugacy relation is hyperfinite on a larger class of Toeplitz subshifts which we call Toeplitz subshifts with growing blocks. This result provides a partial answer to a question asked by Sabok and Tsankov.
\end{abstract}
\maketitle

\section{Introduction}
Descriptive set theory provides a framework to analyze the relative complexity of classification problems from diverse areas of mathematics. Under appropriate coding and identification, various collections of mathematical structures can be naturally regarded as \textit{Polish spaces}, i.e. completely metrizable separable topological spaces. It turns out that many classification problems on these structures can be considered as definable equivalence relations on the corresponding Polish spaces. One can use the notion of \textit{Borel reducibility}, introduced by Friedman and Stanley \cite{FriedmanStanley89} to measure the relative complexity of these definable equivalence relations. For a general development of this framework, we refer the reader to \cite{Gao09}.

Symbolic dynamics has been one of the subjects of this study. In particular, the topological conjugacy relations on various subclasses of subshifts have been extensively analyzed. For example, Clemens \cite{Clemens09} proved that the topological conjugacy relation on subshifts over a finite alphabet is a universal countable Borel equivalence relation. Gao, Jackson, and Seward \cite{GaoJacksonSeward15} analyzed topological conjugacy of generalized $G$-subshifts and showed that topological conjugacy of $G$-subshifts is Borel bireducible with the Borel equivalence relation $E_0$ when $G$ is locally finite; and that topological conjugacy of $G$-subshifts is a universal countable Borel equivalence relation when $G$ is not locally finite. They also proved that the topological conjugacy relation on minimal subshifts over a finite alphabet is not smooth and posed the question of determining the Borel complexity of this relation.

Since then, the project of analyzing the Borel complexity of the topological conjugacy relation for restricted classes of minimal subshifts has been pursued in different directions. For example, Gao and Hill \cite{GaoHill13} have shown that topological conjugacy of minimal rank-1 systems is Borel bireducible with $E_0$. Thomas \cite{Thomas13} proved that the topological conjugacy relation is not smooth for the class of \textit{Toeplitz subshifts}, i.e. minimal subshifts that contain bi-infinite sequences in which every subblock appears periodically.

Subsequent to Thomas' result on Toeplitz subshifts, Sabok and Tsankov \cite{SabokTsankov15} analyzed topological conjugacy of generalized Toeplitz $G$-subshifts for residually finite groups $G$. They proved that topological conjugacy of generalized Toeplitz $G$-subshifts is not hyperfinite if $G$ is residually finite and non-amenable; and that topological conjugacy of Toeplitz subshifts with separated holes is 1-amenable. It is well-known that hyperfiniteness implies 1-amenability \cite[Proposition 2.13]{JacksonKechrisLouveau02} and that 1-amenable relations are hyperfinite $\mu$-almost everywhere for every Borel probability measure $\mu$ on the relevant standard Borel space \cite[Corollary 10.2]{KechrisMiller04}. On the other hand, it is still open whether 1-amenability implies hyperfiniteness.

Sabok and Tsankov asked whether or not the topological conjugacy relation on Toeplitz subshifts is hyperfinite. We will provide a partial affirmative answer to this question and prove the following theorem, which is a strengthening of the result of Sabok and Tsankov.

\begin{theorem}\label{theorem-mainresulttoeplitz} The topological conjugacy relation on Toeplitz subshifts with separated holes over a finite alphabet is hyperfinite.
\end{theorem}

Indeed, we will prove that the topological conjugacy relation is hyperfinite on a larger class of Toeplitz subshifts which we shall call Toeplitz subshifts with growing blocks. Although the class of Toeplitz subshifts with growing blocks is strictly larger than the class of Toeplitz subshifts with separated holes, it does not contain all Toeplitz subshifts and hence the question of whether or not topological conjugacy of Toeplitz subshifts is hyperfinite remains open. Nevertheless, our result can be regarded as an important step towards proving the hyperfiniteness of the topological conjugacy relation on Toeplitz subshifts.

This paper is organized as follows. In Section 2, we will first recall some basic results from the theory of Borel equivalence relations which will be used throughout this paper. Then we will give an overview of Toeplitz sequences and Toeplitz subshifts following \cite{Williams84, Downarowicz05, DownarowiczKwiatkowskiLacroix95} and construct the standard Borel spaces of various subclasses of Toeplitz subshifts. In Section 3, we will prove two lemmas which are
slightly more general restatements of a criterion for Toeplitz subshifts to be topologically
conjugate originally due to Downarowicz, Kwiatkowski, and Lacroix \cite{DownarowiczKwiatkowskiLacroix95}. In Section 4, we will discuss some basic properties of an operation defined on the class of Borel equivalence relations, which is essential to the proof of the main result. In Section 5, we shall prove the main result of this paper. In Section 6, we will briefly describe how our technique may be generalized and discuss further possible research directions.

\section{Preliminaries}
\subsection{Background from the theory of Borel equivalence relations}
In this subsection, we shall discuss some basic notions and results from the theory of Borel equivalence relations.

Suppose that $(X,\mathcal{B})$ is a measurable space, i.e. $\mathcal{B}$ is a $\sigma$-algebra of subsets of $X$. Then $(X,\mathcal{B})$ is said to be a \textit{standard Borel space} if there exists a Polish topology $\tau$ on $X$ such that $\mathcal{B}$ is the Borel $\sigma$-algebra of $(X,\tau)$. It is well-known that if $A \subseteq X$ is a Borel subset of a standard Borel space $(X,\mathcal{B})$, then $(A,\mathcal{B}\upharpoonright A)$ is also a standard Borel space where
\[ \mathcal{B}\upharpoonright A=\{A \cap B: B \in \mathcal{B}\}\]
From now on, while denoting a standard Borel space $(X,\mathcal{B})$, we shall drop the collection of measurable sets and refer to $X$ as a standard Borel space if the standard Borel structure is understood from the context.

Let $X$ and $Y$ be standard Borel spaces. A map $f: X \rightarrow Y$ is called \textit{Borel} if $f^{-1}[B]$ is a Borel subset of $X$ for all Borel subsets $B \subseteq Y$. Equivalently, $f$ is Borel if and only if its graph is a Borel subset of $X \times Y$ where the product $X \times Y$ is endowed with the product $\sigma$-algebra. Two standard Borel spaces $X$ and $Y$ are said to be \textit{(Borel) isomorphic} if there exists a bijection $f: X \rightarrow Y$ such that both $f$ and $f^{-1}$ are Borel. It is a classical result of Kuratowski that any two uncountable standard Borel spaces are isomorphic \cite[Theorem 15.6]{Kechris95}.

An equivalence relation $E \subseteq X \times X$ on a standard Borel space $X$ is called a \textit{Borel equivalence relation} if it is a Borel subset of $X \times X$. Given two Borel equivalence relations $E$ and $F$ on standard Borel spaces $X$ and $Y$ respectively, a Borel map $f: X \rightarrow Y$ is called a \textit{Borel reduction} from $E$ to $F$ if for all $x,y \in X$,
\[ x\mathbin{E}y \Longleftrightarrow f(x)\mathbin{F}f(y)\]
We say that $E$ is \textit{Borel reducible} to $F$, written $E \leq_B F$, if there exists a Borel reduction from $E$ to $F$. Two Borel equivalence relations $E$ and $F$ are said to be \textit{Borel bireducible}, written $E \sim_B F$, if both $E \leq_B F$ and $F \leq_B E$. Finally, we say that $E$ is \textit{strictly less complex} than $F$, written $E <_B F$, if both $E \leq_B F$ and $F \nleq_B E$.

It turns out that there are no $\leq_B$-maximal elements in the $\leq_B$-hierarchy of Borel equivalence relations. In more detail, given a Borel equivalence relation $E$ on a standard Borel space $X$, consider the Borel equivalence relation $E^+$ on the space $X^{\mathbb{N}}$ defined by
\[ x E^+ y \Leftrightarrow \{[x_n]_E: n \in \mathbb{N}\}=\{[y_n]_E: n \in \mathbb{N}\}\]
It is well-known that if $E$ has more than one equivalence class, then $E <_B E^+$ \cite{FriedmanStanley89}. The operation $E \mapsto E^+$ is called the \textit{Friedman-Stanley jump}.

A Borel equivalence relation $E$ is called \textit{countable} (respectively, \textit{finite}) if every equivalence class of $E$ is countable (respectively, finite). Even though there are no $\leq_B$-maximal Borel equivalence relations, if we restrict our attention to countable Borel equivalence relations, then there exists a countable Borel equivalence relation $E_{\infty}$ which is \textit{universal} in the sense that for any countable Borel equivalence relation $F$ we have that $F \leq_B E_{\infty}$. For a detailed development of the theory of countable Borel equivalence relations, we refer the reader to \cite{JacksonKechrisLouveau02}.

A Borel equivalence relation $E$ is said to be \textit{smooth} if it is Borel reducible to the identity relation $\Delta_X$ on some (equivalently, every) uncountable standard Borel space $X$. For example, finite Borel equivalence relations are smooth \cite[Proposition 7.2.1]{Kanovei08}.

A Borel equivalence relation $E$ is said to be \textit{hyperfinite} (respectively, \textit{hypersmooth}) if there is an increasing sequence $F_0 \subseteq F_1 \subseteq F_2 \subseteq \dots$ of finite (respectively, smooth) Borel equivalence relations such that $E=\bigcup_{n \in \mathbb{N}} F_n$. It is easy to check that if $E \subseteq F$ are countable Borel equivalence relations on a standard Borel space $X$ and $F$ is hyperfinite, then $E$ is hyperfinite.

It turns out that hyperfinite Borel equivalence relations are exactly those countable Borel equivalence relations that are Borel reducible to the Borel equivalence relation $E_0$ on the Cantor space $2^\mathbb{N}$ defined by
\[ x E_0 y \Longleftrightarrow \exists m\ \forall n \geq m\ x(n)=y(n)\]
The following theorem of Dougherty, Jackson, and Kechris \cite{DoughertyJacksonKechris94} gives several equivalent characterizations of hyperfiniteness.
\begin{theorem}[Dougherty-Jackson-Kechris]\label{theorem-hyperfiniteequivalent} Let $E$ be a countable Borel equivalence relation on a standard Borel space $X$. Then the following are equivalent.
\begin{itemize}
\item $E$ is hyperfinite.
\item $E$ is hypersmooth.
\item $E \leq E_0$.
\item $E$ is the orbit equivalence relation of a Borel action of $\mathbb{Z}$ on $X$.
\end{itemize}
\end{theorem}

For an example of a hypersmooth Borel equivalence relation which is not countable, consider the Borel equivalence relation $E_1$ on the Polish space $(2^\mathbb{N})^{\natnum}$ defined by
\[ x E_1 y \Longleftrightarrow \exists m\ \forall n \geq m\ x(n)=y(n)\]
It turns out that hypersmooth Borel equivalence relations are exactly those Borel equivalence relations that are Borel reducible to $E_1$ \cite[Proposition 8.1.4]{Gao09}. Combining this fact with Theorem \ref{theorem-hyperfiniteequivalent}, we obtain the following corollary.

\begin{corollary}\label{corollary-maincorollary} If $E$ is a countable Borel equivalence relation such that $E \leq_B E_1$, then $E$ is hyperfinite.
\end{corollary}

\subsection{Background from topological dynamics}

A topological dynamical system is a pair $(X,\varphi)$ where $X$ is a compact metrizable topological space and $\varphi: X \rightarrow X$ is a continuous map. Given a topological dynamical system $(X,\varphi)$, a subset $Y \subseteq X$ is said to be $\varphi$-\textit{invariant} if $\varphi[Y] \subseteq Y$. A \textit{subsystem} of a topological dynamical system $(X,\varphi)$ is a pair of the form $(Y,\varphi)$ where $Y$ is a non-empty closed $\varphi$-invariant subset of $X$.

A topological dynamical system $(X,\varphi)$ is said to be \textit{minimal} if $(X,\varphi)$ has no proper subsystems. Equivalently, $(X,\varphi)$ is minimal if for every $\varphi$-invariant closed subset $Y \subseteq X$ we have either $Y = \emptyset$ or $Y = X$.  Equivalently, $(X,\varphi)$ is minimal if and only if for every $x \in X$ the forward orbit $\orbit(x)={\{\varphi^n(x):n \in \mathbb{N}\}}$ of $x$ is dense in $X$.

A point $x \in X$ in a topological dynamical system $(X,\varphi)$ is said to be \textit{almost periodic} if for every non-empty open neighborhood $U$ of $x$, the set
\[R=\{i \in \mathbb{N}: \varphi^i(x) \in U\}\]
of return times has bounded gaps, i.e. there exists $l \geq 1$ such that for all $n \in \mathbb{N}$
\[R \cap \{n,n+1,\dots,n+l\}\neq \emptyset\]
If $X$ is the forward orbit closure $\shiftorbit(x)$ of some almost periodic point $x \in X$, then $(X,\varphi)$ is minimal. Conversely, if $(X,\varphi)$ is minimal, then every $x \in X$ is almost periodic and has dense orbit \cite[Theorem 2.19]{Kurka03}.

Given two topological dynamical systems $(X,\varphi)$ and $(Y,\psi)$, we say that $(Y,\psi)$ is a \textit{factor} of $(X,\varphi)$ if there exists a continuous surjection $\pi: X \rightarrow Y$ such that
\[\pi \circ \varphi = \psi \circ \pi\]
If $\pi: X \rightarrow Y$ is also a homeomorphism, then $(X,\varphi)$ and $(Y,\psi)$ are said to be \textit{topologically conjugate} and $\pi$ is called a \textit{topological conjugacy}. Similarly, we define the class of \textit{pointed topological dynamical systems} as the class of triples of the form $(X,\varphi,x)$ where $(X, \varphi)$ is a topological dynamical system and $x \in X$. Two pointed systems $(X, \varphi, x)$ and $(Y, \psi, y)$ are said to be \textit{(pointed) topologically conjugate} if there exists a topological conjugacy $\pi: X \rightarrow Y$ between $(X, \varphi)$ and $(Y, \psi)$ such that $\pi(x)=y$.

A topological dynamical system $(X,\varphi)$ is said to be \textit{equicontinuous} if the family $\{\varphi^n: n \in \natnum\}$ of functions is equicontinuous at every point. It is well-known that every topological dynamical system $(X,\varphi)$ admits a \textit{maximal equicontinuous factor} $(Y,\psi)$ in the sense that $(Y,\psi)$ is equicontinuous and any equicontinuous factor of $(X,\varphi)$ is a factor of $(Y,\psi)$. The maximal equicontinuous factor of a topological dynamical system is unique up to topological conjugacy \cite[Theorem 2.44]{Kurka03}.

An \textit{alphabet} is a finite set with at least two elements. For the rest of this paper, fix an alphabet $\alphabet \in \mathbb{N}$. Consider the topological space $\alphabet^{\mathbb{Z}}$ together with the left-shift map $\seqshift: \alphabet^{\mathbb{Z}} \rightarrow \alphabet^{\mathbb{Z}}$ defined by $(\seqshift(\alpha))(i)=\alpha(i+1)$ for all $i \in \mathbb{Z}$ and $\alpha \in \alphabet^{\mathbb{Z}}$. It is easily checked that $(\alphabet^{\mathbb{Z}},\seqshift)$ is a topological dynamical system. A \textit{subshift} over the alphabet $\alphabet$ is a subsystem $(O,\seqshift)$ of the topological dynamical system $(\alphabet^{\mathbb{Z}},\seqshift)$ such that $\seqshift[O]=O$. For notational convenience, we shall often drop the left-shift map $\seqshift$ and refer to $O$ as a subshift. Moreover, we shall exclude the trivial cases and assume henceforth that the underlying topological spaces of subshifts are infinite sets. Thus, a subshift over the alphabet $\alphabet$ is a closed infinite subset of $\alphabet^{\mathbb{Z}}$ which is invariant under both $\seqshift$ and $\seqshift^{-1}$.

A subshift $O \subseteq \alphabet^{\mathbb{Z}}$ is said to be \textit{minimal} if the topological dynamical system $(O,\seqshift)$ is minimal. Being a closed subspace of a Cantor space, any subshift is totally disconnected, compact, and metrizable. If it is also minimal, then it has no isolated points and hence is homeomorphic to a Cantor space. The following well-known theorem \cite[Theorem 6.2.9]{LindMarcus95} characterizes the factor maps between subshifts.

\begin{theorem}[Curtis-Hedlund-Lyndon]\label{curtishedlund} Let $X,Y \subseteq \alphabet^{\mathbb{Z}}$ be subshifts and let
\[\pi: X \rightarrow Y\]
be a continuous map from $X$ to $Y$ commuting with $\sigma$. Then there exist $i \in \mathbb{N}$ and a \textit{block code}, i.e. a function $C: \alphabet^{2i+1} \rightarrow \alphabet$, such that
\[ (\pi(\alpha))(k)=C(\alpha[k-i,k+i])\]
for all $k \in \mathbb{Z}$ and $\alpha \in X$, where $\alpha[k,l]$ denotes the subblock
\[(\alpha(k),\alpha(k+1),\dots,\alpha(l))\]
of the bi-infinite sequence $\alpha$.
\end{theorem}

For any block code $C: \alphabet^{2i+1} \rightarrow \alphabet$, the natural number $i$ called the \textit{length} of the block code $C$ and is denoted by $|C|$. For any topological conjugacy $\pi: X \rightarrow Y$ between subshifts $X$ and $Y$, we define the \textit{length} of $\pi$ to be the natural number
\[|\pi|=max\{min\{|C|: C \text{ induces } \pi\},min\{|C|: C \text{ induces } \pi^{-1}\}\}\]

\subsection{Odometers}

Let $(u_i)_{i \in \mathbb{N}}$ be a sequence of natural numbers such that $(u_i)_{i \in \mathbb{N}}$ is not eventually constant, $u_i > 1$ and $u_i | u_{i+1}$ for all $i \in \mathbb{N}$. Consider the sequence of canonical group homomorphisms
\[\mathbb{Z}_{u_0} \longleftarrow \mathbb{Z}_{u_1} \longleftarrow \mathbb{Z}_{u_2} \cdots\]
where $\mathbb{Z}_{u_i}$ denotes the cyclic group of order $s_i$. Let $\odometer((u_i)_{i \in \mathbb{N}})$ be the inverse limit group
\[ \odometer((u_i)_{i \in \mathbb{N}}):=\varprojlim \mathbb{Z}_{u_i}=\{(m_i) \in \prod_{i \in \mathbb{N}} \mathbb{Z}_{u_i}: m_j \equiv m_i\ (mod\ u_i)\ \text{for } j>i\} \]
with the induced topology. Then the pair $(\odometer((u_i)_{i \in \mathbb{N}}), \odometeraddition)$ is an equicontinuous topological dynamical, where $\odometeraddition(h)=h+\hat{1}$ and $\hat{1}=(1,1,1,\dots)$. Moreover, it is easily checked that $\hat{0}=(0,0,\dots)$ is an almost periodic point with dense orbit. Hence, $(\odometer((u_i)_{i \in \mathbb{N}}), \odometeraddition)$ is minimal. The topological dynamical system $(\odometer((u_i)_{i \in \mathbb{N}}), \odometeraddition)$ is called the \textit{odometer} associated with $(u_i)_{i \in \mathbb{N}}$.

The classification problem for odometers is central to the proof of the main theorem of this paper. We will next recall some basic results regarding the classification of odometers up to topological conjugacy. For a detailed survey of odometers, we refer the reader to \cite{Downarowicz05}.

A \textit{supernatural number} is a formal product $\prod_{i \in \natnump} \textrm{p}_i^{k_i}$ where $\textrm{p}_i$ is the $i$-th prime number and $k_i \in \mathbb{N} \cup \{\infty\}$ for all $i \in \mathbb{N}$. For each sequence $(u_i)_{i \in \mathbb{N}}$ of positive integers, define $\lcmprod(u_i)_{i \in \mathbb{N}}$ to be the supernatural number $\mathbf{u}=\prod_{i \in \natnump} \textrm{p}_i^{k_i}$ where
\[ k_i=sup\{j \in \mathbb{N}: \exists m \in \mathbb{N}\ \textrm{p}_i^j | u_m \} \]
Given a supernatural number $\mathbf{u}$, any sequence $(u_i)_{i \in \mathbb{N}}$ such that $\lcmprod(u_i)_{i \in \mathbb{N}}=\mathbf{u}$ and $u_i | u_{i+1}$ for all $i \in \mathbb{N}$ will be called a \textit{factorization} of $\mathbf{u}$ and any positive integer $q$ dividing some $u_i$ will be called a \textit{factor} of $\mathbf{u}$. It turns out that the set of supernatural numbers is a complete set of invariants for topological conjugacy of odometers and hence the topological conjugacy problem for odometers is smooth.

\begin{theorem}\cite{BuescuStewart95}\label{theorem-conjugacyodometer} The odometers $(\odometer((u_i)_{i \in \mathbb{N}}),\odometeraddition)$ and $(\odometer((v_i)_{i \in \mathbb{N}}), \odometeraddition)$ are topologically conjugate if and only if $\lcmprod(u_i)_{i \in \mathbb{N}}=\lcmprod(v_i)_{i \in \mathbb{N}}$.
\end{theorem}

\subsection{Toeplitz sequences and Toeplitz subshifts}
In this subsection, we will give a detailed overview of Toeplitz sequences and Toeplitz subshifts following \cite{Williams84, Downarowicz05, DownarowiczKwiatkowskiLacroix95}.

A bi-infinite sequence $\alpha \in \shiftspace$ is called a \textit{Toeplitz sequence} over the alphabet $\alphabet$ if for all $i \in \mathbb{Z}$ there exists $j \in \natnump$ such that $\alpha(i+kj)=\alpha(i)$ for all $k \in \mathbb{Z}$. Equivalently, Toeplitz sequences are those in which every subblock appears periodically. Periodic sequences are obviously Toeplitz. However, we shall exclude these since we are interested in infinite subshifts generated by Toeplitz sequences and periodic sequences have finite orbits under $\seqshift$. From now on, all Toeplitz sequences are assumed to be non-periodic unless stated otherwise.

In our analysis of the structure of Toeplitz sequences, we will need the following objects associated to each sequence $\alpha \in \shiftspace$ for each $p \in \mathbb{N}^+$.

\begin{itemize}
\item The \textit{p-periodic parts} of $\alpha$ is defined to be the set of indices
\[ Per_p(\alpha):= \bigcup_{a \in \alphabet} Per_p(\alpha,a)\]
where $Per_p(\alpha,a):=\{i \in \mathbb{Z}: \forall k \in \mathbb{Z}\ \alpha(i+pk)=a\}$ for each symbol $a \in \alphabet$. In other words,
\[Per_p(\alpha)=\{i \in \mathbb{Z}: \forall k \in \mathbb{Z}\ \alpha(i)=\alpha(i+pk)\}\]
$p$ is called a \textit{period} of $\alpha$ if $Per_p(\alpha)\neq \emptyset$. It follows from the definitions that the sequence $\alpha$ is a Toeplitz sequence if and only if $\bigcup_{p \in \natnump}Per_p(\alpha) = \mathbb{Z}$.
\item The sequence obtained from $\alpha$ by replacing $\alpha(i)$ with the blank symbol $\blanksymbol$ for each $i \notin Per_p(\alpha)$ will be called the \textit{p-skeleton} of $\alpha$. The $p$-skeleton of $\alpha$ will be denoted by $Skel(\alpha,p)$.
\item Any subblock of the $p$-skeleton of $\alpha$ which consists of non-blank symbols and which is preceded and followed by a blank symbol will be called a \textit{filled p-block} of the $p$-skeleton of $\alpha$.
\item The indices of the $p$-skeleton of $\alpha$ containing the blank symbol will be called the \textit{p-holes} of $\alpha$.
\item The set of \textit{p-symbols} of $\alpha$ is the set of words
\[W_p(\alpha)=\{\alpha[kp,(k+1)p):k \in \mathbb{Z}\}\]
\end{itemize}

Let $p,q \in \natnump$ be periods of some sequence $\alpha \in \shiftspace$. It easily follows from the definitions that $Per_p(\alpha) \subseteq Per_q(\alpha)$ whenever $p|q$; and that
\[Per_{gcd(p,q)}(\alpha)=Per_p(\alpha)\]
whenever $Per_p(\alpha) \subseteq Per_q(\alpha)$. A positive integer $p \in \natnump$ is an \textit{essential period} of $\alpha$ if $p$ is a period of $\alpha$ and for all $q < p$ we have $Per_p(\alpha) \neq Per_q(\alpha)$. Equivalently, $p$ is an essential period of $\alpha$ if and only if the $p$-skeleton of $\alpha$ is not periodic with any smaller period. It can easily be checked that if $p$ and $q$ are essential periods of $\alpha$, then so is $lcm(p,q)$. Thus we can associate a supernatural number to each sequence whose set of periods is non-empty by taking the least common multiple of all essential periods.

The \textit{scale} of a Toeplitz sequence $\alpha$ is the supernatural number $\mathbf{u}_{\alpha}=\lcmprod(u_i)_{i \in \mathbb{N}}$ where $u_i$ is an enumeration of the essential periods of $\alpha$.

Every subblock of a Toeplitz sequence $\alpha$ appears periodically along $\alpha$ and hence the return times of $\alpha$ to any basic clopen subset of its shift orbit closure $\shiftorbit(\alpha)$ contains an infinite progression of the form $p+q \mathbb{Z}$. It follows that $\alpha$ is an almost periodic point of $(\shiftorbit(\alpha),\seqshift)$ and hence $\shiftorbit(\alpha)$ is a minimal subshift. A subshift $O$ is said to be a \textit{Toeplitz subshift} over the alphabet $\alphabet$ if $O=\shiftorbit(\alpha)$ for some Toeplitz sequence $\alpha \in \shiftspace$.

We will next prove that the maximal equicontinuous factor of a Toeplitz subshift $\shiftorbit(\alpha)$ is the odometer associated to the supernatural number $\mathbf{u}_{\alpha}$. The following results and the construction of the maximal equicontinuous factor are originally due to Williams \cite{Williams84}. We remark that even though the statements of the following lemmas are more general than Williams' original results, they can be proved with the same proofs.

\begin{lemma}\cite{Williams84}\label{lemma-williamslemma0} Let $p \in \natnump$ and for each $0 \leq k < p$, define 
\[A(\alpha,p,k):=\{\seqshift^i(\alpha): k \equiv i\ (mod\ p)\}\]
Then each element of $\overline{A(\alpha,p,k)}$ has the same $p$-skeleton as $\seqshift^k(\alpha)$, i.e. for each $a \in \alphabet$, we have that $Per_{p}(\seqshift^k(\alpha),a) = Per_{p}(\gamma,a)$ for all $\gamma \in \overline{A(\alpha,p,k)}$.
\end{lemma}
\begin{lemma}\cite{Williams84} \label{lemma-williamslemma1} Let $(r_i)_{i \in \mathbb{N}}$ be a factorization of $\mathbf{u}_{\alpha}$ and let $A(\alpha,r_i,k)$ be defined as in Lemma \ref{lemma-williamslemma0}. For each $i \in \mathbb{N}$ and $0 \leq k < r_i$, we have that
\begin{itemize}
\item[a.] $\{\overline{A(\alpha,r_i,k)}: 0 \leq k < r_i\}$ is a partition of $\shiftorbit(\alpha)$.
\item[b.] $\overline{A(\alpha,r_i,k)} \subseteq \overline{A(\alpha,r_j,l)}$ for all $j < i$ and $k \equiv l\ (mod\ r_j)$.
\item[c.] $\seqshift[\overline{A(\alpha,r_i,r_i-1)}]=\overline{A(\alpha,r_i,0)}$ and $\seqshift[\overline{A(\alpha,r_i,k)}]=\overline{A(\alpha,r_i,k+1)}$ for all $0 \leq k < r_i-1$.
\end{itemize}
\end{lemma}

We note that by Lemma \ref{lemma-williamslemma0} and Lemma \ref{lemma-williamslemma1}.a, any essential period of $\alpha$ is an essential period of any $\gamma \in \shiftorbit(\alpha)$ and vice versa. Therefore, it makes sense to define the \textit{scale} of a Toeplitz subshift $O$ to be the supernatural number that is the least common multiple of all essential periods of some (equivalently, every) point of $O$.

Consider the map $\psi : \shiftorbit(\alpha) \rightarrow \odometer(u_i)_{i \in \mathbb{N}}$ given by
\[ \psi(x)=(m_i)_{i \in \mathbb{N}} \]
where $x \in \overline{A(\alpha,u_i,m_i)}$ and $\lcmprod(u_i)_{i \in \mathbb{N}}=\mathbf{u}_{\alpha}$. It is not difficult to check that the map $\psi$ is continuous. Moreover, we have that $\psi \circ \seqshift = \lambda \circ \psi$ and hence $\psi$ is a factor map. In order to show that $\odometer(u_i)_{i \in \mathbb{N}}$ is the maximal equicontinuous factor of $\shiftorbit(\alpha)$, it is sufficient to prove that $\psi^{-1}[\psi(\alpha)]=\{\alpha\}$. (For example, see \cite[Proposition 1.1]{Paul76}.)

Recall by Lemma \ref{lemma-williamslemma0} that the $u_i$-skeletons of the sequences in the set $\overline{A(\alpha,u_i,m_i)}$ are the same. Hence two sequences $\beta, \beta' \in \shiftorbit(\alpha)$ have the same $u_k$-skeleton whenever $\psi(\beta) \upharpoonright k+1 = \psi(\beta') \upharpoonright k+1$. This implies that $\psi$ is one to one on the set of Toeplitz sequences since every subblock of a Toeplitz sequence eventually appears in some $u_k$-skeleton. In particular, we have that $\psi^{-1}[\psi(\alpha)]=\{\alpha\}$, which completes the proof that $\odometer(u_i)_{i \in \mathbb{N}}$ is the maximal equicontinuous factor of $\shiftorbit(\alpha)$.

Recall that the maximal equicontinuous factor of a topological dynamical system is unique up to topological conjugacy. Consequently, it follows from Theorem \ref{theorem-conjugacyodometer} that topologically conjugate Toeplitz subshifts have the same scale.

\subsection{Various subclasses of Toeplitz subshifts}

Given a Toeplitz sequence $\alpha$ and a factorization $(u_i)_{i \in \mathbb{N}}$ of its scale $\mathbf{u}_{\alpha}$, we can imagine $\alpha$ to be obtained by a recursive construction where we start the construction with the two-sided constant sequence of blank symbols and replace the blank symbols corresponding to the indices $Per_{u_i}(\alpha)$ periodically with the appropriate symbols at the $i$-th stage. This way of understanding Toeplitz sequences from their constructions allows us to isolate some special types of Toeplitz sequences as considered by Downarowicz \cite[Section 9]{Downarowicz05}.

Of particular interest in this thesis will be the class of Toeplitz subshifts with separated holes. A Toeplitz subshift $O$ is said to have \textit{separated holes} with respect to $(u_i)_{i \in \mathbb{N}}$ if the minimal distance between the $u_i$-holes in the $u_i$-skeleton of every (equivalently, some) element of $O$ grows to infinity with $i$, where $(u_i)_{i \in \mathbb{N}}$ is a factorization of the scale of $O$.

It turns out that whether or not a Toeplitz subshift has separated holes is independent of the particular factorization $(u_i)_{i \in \mathbb{N}}$. Let $(u_i)_{i \in \mathbb{N}}$ and $(v_i)_{i \in \mathbb{N}}$ be two factorizations of the same supernatural number $\mathbf{u}$ and let $O$ be a Toeplitz subshift with scale $\mathbf{u}$. It is easily checked that $O$ has separated holes with respect to $(u_i)_{i \in \mathbb{N}}$ if and only if $O$ has separated holes with respect to $(v_i)_{i \in \mathbb{N}}$

We will next define a property that generalizes the property of having separated holes. Given a Toeplitz subshift $O$ and a Toeplitz sequence $\alpha \in O$, let $A(\alpha,p,k)$ be defined as in Lemma \ref{lemma-williamslemma0}. Notice that for any $\beta \in \shiftorbit(\alpha)$, regardless of whether or not $\beta$ is a Toeplitz sequence, we have that
\[ \{\overline{A(\beta,u_i,k)}: 0 \leq k < u_i\}=\{\overline{A(\alpha,u_i,k)}: 0 \leq k < u_i\}\]
since the orbit of $\beta$ is dense in $\shiftorbit(\alpha)$ by minimality. Therefore, this partition only depends on $u_i$ and it will be denoted by $Parts(\shiftorbit(\alpha),u_i)$. Moreover, every element of $\overline{A(\alpha,u_i,k)}$ has the same $u_i$-skeleton by Lemma \ref{lemma-williamslemma0}. Consequently, for each $W \in Parts(\shiftorbit(\alpha),u_i)$, we can define the $u_i$-\textit{skeleton} of $W$ to be the $u_i$-skeleton of some (equivalently, every) element of $W$ and denote it by $Skel(W,u_i)$. Define $Parts_*(O,u_i)$ to be the set
\[\{W \in Parts(O,u_i): Skel(W,u_i)(0) \neq \blanksymbol\ \wedge\ Skel(W,u_i)(-1) = \blanksymbol\}\]
For each $W \in Parts_*(O,u_i)$, let $length(W)$ be the smallest positive integer such that $Skel(u_i,W)(length(W))=\blanksymbol$. In other words, $length(W)$ is the length of the filled $u_i$-block of the $u_i$-skeleton of $W$ whose first non-blank symbol is positioned at index $0$. $O$ is said to have \textit{growing blocks} with respect to $(u_i)_{i \in \mathbb{N}}$ if
\[ \displaystyle \lim_{i \rightarrow \infty} min \{ length(W): W \in Parts_*(O,u_i)\}=+\infty\]
i.e., $O$ has growing blocks with respect to $(u_i)_{i \in \mathbb{N}}$ if the minimal length of filled $u_i$-blocks grows to infinity with $i$.

Recall that $Per_p(\alpha) \subseteq Per_q(\alpha)$ whenever $p|q$. It follows that if a Toeplitz subshift $O$ has separated holes with respect to some factorization of its scale $\mathbf{u}$, then it has separated holes with respect to any factorization of $\mathbf{u}$ and hence it has growing blocks with respect to any factorization of $\mathbf{u}$. Unlike having separated holes, having growing blocks is not independent of the particular factorization $(u_i)_{i \in \mathbb{N}}$. Consider the Toeplitz sequence whose $(2^k 5)$-skeletons restricted to the interval $[0,2^k 5)$ are given by
\begin{align*}
& 0 \blanksymbol \blanksymbol \blanksymbol 0 \\
& 0 \blanksymbol 1 \blanksymbol 0 0 \blanksymbol \blanksymbol \blanksymbol 0 \\
& 0 0 1 0 0 0 \blanksymbol \blanksymbol \blanksymbol 0 0 1 1 1 0 0 1 0 1 0 \\
& 0 0 1 0 0 0 \blanksymbol 1 \blanksymbol 0 0 1 1 1 0 0 1 0 1 0 0 0 1 0 0 0 \blanksymbol \blanksymbol \blanksymbol 0 0 1 1 1 0 0 1 0 1 0 \\
& \dots
\end{align*}
for each $k \in \mathbb{N}$. We initially start with the $5$-skeleton consisting of the repeated blocks $0 \blanksymbol \blanksymbol \blanksymbol 0$. At every odd stage $k$, we fill the hole in the middle of the leftmost $\blanksymbol \blanksymbol \blanksymbol$ block along each interval $[j2^k 5,(j+1)2^k 5)$ with the symbol $1$. At every even stage $k$, along each interval $[j2^k 5,(j+1)2^k 5)$, we fill the first two single holes with the symbol $0$, the remaining single holes with the symbol $1$, and replace the rightmost $\blanksymbol \blanksymbol \blanksymbol$ block by the block $101$. It is easily checked that the Toeplitz subshift generated by this Toeplitz sequence does not have growing blocks with respect to $(2^k 5)_{k \in \mathbb{N}}$. However, it does have growing blocks with respect to $(4^k 5)_{k \in \mathbb{N}}$.

\subsection{The standard Borel space of Toeplitz subshifts}

In this subsection, we will construct the standard Borel spaces of various subclasses of Toeplitz subshifts over the alphabet $\alphabet$.

The set $K(X)$ of non-empty compact subsets of a Polish space $X$ endowed with the topology induced by the Hausdorff metric is a Polish space \cite[Section 4.F]{Kechris95}. It is not difficult to check that the map $C \mapsto \seqshift[C]$ is a homeomorphism of $K(\shiftspace)$ and that the set
\[ \mathcal{S}_{\alphabet}=\{O \subseteq \shiftspace: O \text{ is a subshift}\} \]
of subshifts over the alphabet $\alphabet$ is a Borel subset of $K(\shiftspace)$ and hence is a standard Borel space \cite[Lemma 3]{Clemens09}. Recall that the following are equivalent for a subshift.
\begin{itemize}
\item[a.] $(O,\seqshift)$ is minimal.
\item[b.] For every $x \in O$, $x$ is almost periodic and $O=\shiftorbit(x)$.
\item[c.] For some $x \in O$, $x$ is almost periodic and $O=\shiftorbit(x)$.
\end{itemize}
There exists a sequence of Borel functions which select a dense set of points from each element of $K(\shiftspace)$ \cite[Theorem 12.13]{Kechris95} and hence we can check in a Borel way whether or not a compact subset of $\shiftspace$ is the closure of the orbit of an almost periodic point. It follows that the set $\mathcal{M}_{\alphabet}$ of minimal subshifts is a Borel subset of $\mathcal{S}_{\alphabet}$ and hence is a standard Borel space.

In order to construct the standard Borel space of Toeplitz subshifts, we will need the following theorem regarding the definability of Baire category notions.
\begin{theorem}\cite{SabokTsankov15} \label{theorem-nonmeagerlymany} Let $X$ be a Polish space and let $F(X)$ be the Effros Borel space $F(X)$ consisting of closed subsets of $X$. Then for any Borel subset $A \subseteq X$, the set
\[ \{F \in F(X): \exists^* x \in F\ x \in A\} \]
is Borel, where the quantifier $\exists^* x \in F$ stands for ``For non-meagerly many $x$ in $F$".
\end{theorem}

Since the set of Toeplitz sequences in a Toeplitz subshift form a dense $G_{\delta}$ subset \cite[Theorem 5.1]{Downarowicz05} and the set of Toeplitz sequences is a Borel subset of $\shiftspace$, it follows from Theorem \ref{theorem-nonmeagerlymany} that the set
\[ \spaceoftoeplitzsubshift := \{O \in \mathcal{M}_{\alphabet}: O \text{ is a Toeplitz subshift}\} \]
is a Borel subset of $\mathcal{M}_{\alphabet}$ and hence is a standard Borel space.

We will next construct the standard Borel spaces of Toeplitz subshifts with growing blocks and separated holes. However, since having growing blocks is not independent of the factorization we use for each supernatural number, in order to construct the standard Borel space of Toeplitz subshifts with growing blocks, we need to fix a map that assigns a factorization to each supernatural number. Moreover, we want to express the property of having growing blocks with a Borel condition and hence the factorization map we will use should be Borel when considered as a function from $(\mathbb{N} \cup \{\infty\})^{\mathbb{N}}$ to $(\natnump)^{\mathbb{N}}$. Given a supernatural number $\mathbf{r}=\prod_{i \in \natnump} \textrm{p}_i^{k_i}$, let
\[\dot{r}_t=\prod_{1\leq i\leq t+1} \textrm{p}_i^{min\{k_i,t+1\}}\]
and define the natural factorization $(r_t)_{t \in \mathbb{N}}$ of $\textbf{r}$ to be the sequence obtained from the sequence $(\dot{r}_t)_{t \in \mathbb{N}}$ by deleting all $1$'s and the repeated terms. We note that all results in this paper hold for any Borel factorization of supernatural numbers.

Now fix a Borel map that chooses a point from each element of $\spaceoftoeplitzsubshift$. Since all points in Toeplitz subshifts have the same essential periods, we can construct a Borel map from $\tau: \spaceoftoeplitzsubshift \rightarrow (\natnump)^{\mathbb{N}}$ that sends each Toeplitz subshift to the natural factorization of its scale. By Lemma \ref{lemma-williamslemma0} and Lemma \ref{lemma-williamslemma1}.a, the $p$-skeleton structures of all points in a Toeplitz subshift are the same, up to shifting. Moreover, both having separated holes and growing blocks with respect to the natural factorization can be expressed by Borel conditions. Thus both
\[ \spaceoftoeplitzsubshiftsep := \{O \in \spaceoftoeplitzsubshift: O \text{ has separated holes}\} \]
and
\[ \spaceoftoeplitzsubshiftgr := \{O \in \spaceoftoeplitzsubshift: O \text{ has growing blocks with respect to } \tau(O)\} \]
are Borel subsets of $\spaceoftoeplitzsubshift$ and hence are standard Borel spaces. The topological conjugacy relations on the standard Borel spaces $\spaceoftoeplitzsubshift$, $\spaceoftoeplitzsubshiftsep$, and $\spaceoftoeplitzsubshiftgr$ are clearly countable Borel equivalence relations. Moreover, it follows from the work of Thomas \cite{Thomas13} that $E_0$ is Borel reducible to the topological conjugacy relation on $\spaceoftoeplitzsubshiftsep$.

\section{Topological conjugacy of Toeplitz subshifts}

Downarowicz, Kwiatkowski, and Lacroix found a criterion for Toeplitz subshifts to be topologically conjugate in \cite{DownarowiczKwiatkowskiLacroix95}. In the proof of Theorem \ref{theorem-mainresulttoeplitz}, we will need this criterion in a slightly more general form than it was originally formulated. In this section, we will include these more general statements with their proofs. We note that all results in this section are extracted from \cite[Theorem 1]{DownarowiczKwiatkowskiLacroix95}.

\begin{lemma}\label{lemma-obscurelemma} Let $O$ and $O'$ be Toeplitz subshifts over the alphabet $\alphabet$ and let $\pi: O \rightarrow O'$ be a topological conjugacy such that $\pi(\alpha)=\beta$ for $\alpha \in O$ and $\beta \in O$. Then for any $p \in \natnump$ such that $[-|\pi|,|\pi|] \subseteq Per_{p}(\alpha),Per_{p}(\beta)$ there exists $\phi \in Sym(\alphabet^p)$ such that
\[ \phi(\alpha[kp,(k+1)p))=\beta[kp,(k+1)p)\]
for all $k \in \mathbb{Z}$.
\end{lemma}
\begin{proof} Let $p \in \natnump$ be such that $[-|\pi|,|\pi|] \subseteq Per_{p}(\alpha), Per_{p}(\beta)$. Consider the relation $\Gamma: W_p(\alpha) \rightarrow W_p(\beta)$ given by
\[ \Gamma(\alpha[kp,(k+1)p))=\beta[kp,(k+1)p)\]
for each $k \in \mathbb{Z}$. We want to prove that $\Gamma$ is well-defined and one to one. Pick $k,k' \in \mathbb{Z}$ such that $\alpha[kp,(k+1)p)=\alpha[k'p,(k'+1)p)$. Since $[-|\pi|,|\pi|] \subseteq Per_{p}(\alpha), Per_{p}(\beta)$ we have that
\[\alpha[kp-|\pi|,(k+1)p+|\pi|]=\alpha[k'p-|\pi|,(k'+1)p+|\pi|]\]
By the definition of $|\pi|$, there exists some block code $C$ inducing $\pi$ such that $|C| \leq |\pi|$. Then we have that
\begin{align*}
\beta(kp+u) &= (\pi(\alpha))(kp+u)\\
&= C(\alpha[kp+u-|C|,kp+u+|C|])\\
&= C(\alpha[k'p+u-|C|,k'p+u+|C|])\\
&= (\pi(\alpha))(k'p+u)\\
&= \beta(k'p+u)
\end{align*}
for any $0 \leq u < p$ and hence $\beta[kp,(k+1)p)=\beta[k'p,(k'+1)p)$. This proves that $\Gamma$ is well-defined. Since there exists a block code $C'$ inducing $\pi^{-1}$ such that $|C'| \leq |\pi|$, a symmetrical argument shows that $\Gamma$ is one to one. It follows that $\Gamma$ is a bijection and hence we can choose $\phi \in Sym(\alphabet^p)$ to be any permutation extending $\Gamma$.
\end{proof}

\begin{lemma}\label{lemma-notobscurelemma} Let $O$ and $O'$ be Toeplitz subshifts with the same scale $\textbf{r}$. Assume that there exist a factor $p$ of $\textbf{r}$ and $\phi \in Sym(\alphabet^p)$ such that
\[ \phi(\alpha[kp,(k+1)p))=\beta[kp,(k+1)p) \text{ for all } k \in \mathbb{Z}\]
for some points $\alpha \in O$ and $\beta \in O'$. Then $(O,\seqshift,\alpha)$ and $(O',\seqshift,\beta)$ are pointed topologically conjugate.
\end{lemma}
\begin{proof} Observe that $\phi$ induces a homeomorphism $\widehat{\phi}$ of $\shiftspace$ defined by
\[\widehat{\phi}(\gamma)[kp,(k+1)p)=\phi(\gamma[kp,(k+1)p))\]
for all $k \in \mathbb{Z}$ and $\gamma \in \shiftspace$. Obviously
\[\widehat{\phi}(\seqshift^{pk}(\alpha))=\seqshift^{pk}(\beta)\]
for any $k \in \mathbb{Z}$. Let $A(\alpha,p,0)$ and $A(\beta,p,0)$ be defined as in Lemma \ref{lemma-williamslemma0}. Since $\widehat{\phi}$ is a homeomorphism and $(\widehat{\phi})^{-1}=\widehat{\phi^{-1}}$, it easily follows that
\[\widehat{\phi}[\overline{A(\alpha,p,0)}]=\overline{A(\beta,p,0)}\]
Recall that $\{\overline{A(\alpha,p,k)}: 0 \leq k < p\}$ and $\{\overline{A(\beta,p,k)}: 0 \leq k < p\}$ are partitions of $O$ and $O'$ respectively. Let $\pi$ be the map from $O$ to $O'$ given by
\[ \pi(\gamma)=\seqshift^{i}(\widehat{\phi}(\seqshift^{-i}(\gamma))) \text{ if } \gamma \in \overline{A(\alpha,p,i)} \]
Obviously $\pi$ is a bijection between $O$ and $O'$. Moreover, it is continuous on each $\overline{A(\alpha,p,i)}$. Since the sets $\overline{A(\alpha,p,i)}$ are at a positive distant apart from each other, it follows that $\pi$ is continuous on $O$ and hence is a homeomorphism between $O$ and $O'$. We want to show that $\pi$ is shift preserving. For any $0 \leq i < p-2$ and for any $\gamma \in \overline{A(\alpha,p,i)}$, we have that
\[ \pi(\seqshift(\gamma))=\seqshift^{i+1}(\widehat{\phi}(\seqshift^{-(i+1)}(\seqshift(\gamma))))=\seqshift(\seqshift^{i}(\widehat{\phi}(\seqshift^{-i}(\gamma))))=\seqshift(\pi(\gamma)) \]
Since $\widehat{\phi}$ commutes with $\seqshift^{p}$, for any $\gamma \in \overline{A(\alpha,p,p-1)}$ we have that
\begin{align*}
\seqshift(\pi(\gamma))&=\seqshift(\seqshift^{(p-1)}(\widehat{\phi}(\seqshift^{-(p-1)}(\gamma))))\\
&=\seqshift^p(\widehat{\phi}(\seqshift^{-(p-1)}(\gamma)))\\
&= \widehat{\phi}(\seqshift^p(\seqshift^{-(p-1)}(\gamma)))\\
&=\widehat{\phi}(\seqshift(\gamma))=\pi(\seqshift(\gamma))
\end{align*}
Therefore, $\pi$ is a topological conjugacy between $O$ and $O'$ sending $\alpha$ to $\beta$.
\end{proof}

We remark that the proofs of Lemma \ref{lemma-obscurelemma} and Lemma \ref{lemma-notobscurelemma} together imply that if $O$ and $O'$ are topologically conjugate Toeplitz subshifts, then some elements of the partition $Parts(O,p)$ are mapped onto some elements of the partition $Parts(O',p)$ under the natural action of $Sym(n^p)$ for a sufficiently large factor $p$ of the common scale.

\section{Restricting the Friedman-Stanley jump to finite subsets}

Recall that the set $K(X)$ of non-empty compact subsets of a Polish space $X$ is a Polish space endowed with the topology induced by the Hausdorff metric. It is easily checked that the set
\[\Fin(X):=\{F \subseteq X: F \text{ is finite and non-empty}\}\]
is an $F_{\sigma}$ subset of $K(X)$ and hence is a standard Borel space. Given a Borel equivalence relation $E$ on a standard Borel space $X$, let $E^{\fin}$ be the equivalence relation on $\Fin(X)$ defined by
\[ u\ E^{\fin}\ v \Leftrightarrow \{[x]_E: x \in u\}=\{[x]_E: x \in v\}\]
It is routine to check that $E^{\fin}$ is a Borel equivalence relation. Even though $E^{\fin}$ is not a subrelation of the Friedman-Stanley jump $E^+$, we can think of $E^{\fin}$ as the restriction of $E^+$ to the finite subsets of $X$. (It is not difficult to show that $E^{\fin}$ is Borel bireducible with the restriction of $E^+$ to the Borel subset of $X^{\mathbb{N}}$ consisting of sequences in which only finitely many elements of $X$ appear.)

We will now explore some basic properties of the map $E \mapsto E^{\fin}$. We begin by noting that the Borel map $x \mapsto \{x\}$ is a Borel reduction from $E$ to $E^{\fin}$ for every Borel equivalence relation $E$ and that if $f: X \rightarrow Y$ is a Borel reduction witnessing $E\leq_B F$, then $u \mapsto f[u]$ is a Borel reduction from $E^{\fin}$ to $F^{\fin}$. It is easily checked that if $E$ is a finite (respectively, countable) Borel equivalence relation, then $E^{\fin}$ is also a finite (respectively, countable) Borel equivalence relation; and that $E^{\fin}$ is smooth whenever $E$ is smooth. Moreover, the map $E \mapsto E^{\fin}$ commutes with increasing unions, i.e. if $E_0 \subseteq E_1 \subseteq \dots $ is an increasing sequence of Borel equivalence relations on a standard Borel space $X$, then $E_0^{\fin} \subseteq E_1^{\fin} \subseteq \dots $ is an increasing sequence of Borel equivalence relations on $\Fin(X)$ and
\[\bigcup_{i \in \mathbb{N}} E_i^{\fin}=E^{\fin} \text{ where }E=\bigcup_{i \in \mathbb{N}} E_i\]
Consequently, if $E$ is a hyperfinite (respectively, hypersmooth) Borel equivalence relation, then $E^{\fin}$ is also hyperfinite (respectively, hypersmooth).

It is well-known \cite{Silver80, KechrisLouveau97, HarringtonKechrisLouveau90} that there are no $\leq_B$-intermediate Borel equivalence relations between the consecutive pairs of the sequence of Borel equivalence relations
\[\Delta_{\mathbb{N}} <_B \Delta_{\mathbb{R}} <_B E_0 <_B E_1 \]
Moreover, the Borel equivalence relations $\Delta_{\mathbb{N}}$, $\Delta_{\mathbb{R}}$, $E_0$, $E_1$, and $E_{\infty}$ are fixed points of the map $E \mapsto E^{\fin}$ up to Borel bireducibility. Based on this observation, one might conjecture that $E \sim_B E^{\fin}$ for all Borel equivalence relations $E$ with infinitely many $E$-classes. However, this naive conjecture turns out to be false. As we shall see, $E^{\fin}$ behaves like a universal finite index extension of $E$ for every countable Borel equivalence relation $E$ and not every countable Borel equivalence relation is Borel bireducible with all of its finite index extensions. We first need to recall some basic definitions.

Let $E \subseteq F$ be countable Borel equivalence relations on a standard Borel space $X$. Then $F$ is called a \textit{finite index extension} of $E$ if every $F$-class consists of finitely many $E$-classes. We will write $[F:E]<\infty$ to denote that $F$ is a finite index extension of $E$.

\begin{proposition} \label{theorem-universalfiniteindex} Let $E$ be a countable Borel equivalence relation on a standard Borel space $X$. Then for every countable Borel equivalence relation $F$ on $X$ with $[F:E]<\infty$, we have that $F \leq_B E^{\fin}$.
\end{proposition}
\begin{proof} By the Feldman-Moore theorem \cite[Theorem 7.1.4]{Gao09}, there exists a countable discrete group $G$ such that $F$ is the orbit equivalence relation of a Borel action of $G$ on $X$. Let $(g_i)_{i \in \mathbb{N}}$ be a fixed enumeration of elements of $G$ and consider the map $h: X \rightarrow \Fin(X)$ defined by
\[ x \mapsto \{g_i \cdot x: i \leq j_x\} \]
where $j_x$ is the least natural number such that
\[\forall k\ \exists i \leq j_x\ (g_k \cdot x, g_i \cdot x) \in E\]
It is easily checked that $h$ is a Borel map. Notice that $h$ maps every $x$ to a finite subset of $X$ that contains representatives of each $E$-class contained in $[x]_F$. Hence $x F y \Leftrightarrow h(x) E^{\fin} h(y)$ for all $x,y \in X$.
\end{proof}

It follows that if $E \subseteq F$ is a pair of countable Borel equivalence relations such that $[F:E] < \infty$ and $F \nleq_B E$, then $E^{\fin} \nleq_B E$ and hence $E <_B E^{\fin}$. It is well-known that such pairs of countable Borel equivalence relations exist \cite{Adams02}. One may ask whether or not the only obstacle for a countable Borel equivalence relation $E$ to satisfy $E \sim_B E^{\fin}$ is the existence of such a finite index extension.

\begin{question} Let $E$ be a countable Borel equivalence relation on a standard Borel space $X$ such that $E^{\fin} \nleq_B E$. Does there necessarily exist a countable Borel equivalence relation $F$ such that $[F:E] < \infty$ and $F \nleq_B E$?
\end{question}

\section{Proof of the main result}

In this section, we shall prove Theorem \ref{theorem-mainresulttoeplitz}. Indeed, we will prove the stronger result that the topological conjugacy relation on the standard Borel space $\spaceoftoeplitzsubshiftgr$ of Toeplitz subshifts with growing blocks is hyperfinite.

Before we present the proof of the main result of this chapter, we will prove the following easy but useful proposition which shows that if there exists a $p$-hole in the $p$-skeleton of a sequence $\alpha$, then there exists a $p$-hole in the $p$-skeleton of its image under a block code, which is no further from the $p$-hole in the sequence $\alpha$ than the length of the block code.

\begin{proposition}\label{proposition-phole} Let $(O,\seqshift,\alpha)$ and $(O',\seqshift,\beta)$ be pointed Toeplitz subshifts and let $\pi$ be a factor map from $O$ onto $O'$ such that $\pi(\alpha)=\beta$. Assume that $m \in \natnump$ is the length of some block code $C$ inducing $\pi$. Then for all $p \in \natnump$ and $k \in \mathbb{Z}$, we have that $k \in Per_p(\beta)$ whenever $[k-m,k+m] \subseteq Per_p(\alpha)$.
\end{proposition}
\begin{proof}
For all $p \in \natnump$ and $k \in \mathbb{Z}$, if $[k-m,k+m] \subseteq Per_p(\alpha)$, then for all $l \in \mathbb{Z}$
\begin{align*}
\beta(k+pl)=(\pi(\alpha))(k+pl)&=C(\alpha[k+pl-m,k+pl+m])\\
&=C(\alpha[k-m,k+m])\\
&=(\pi(\alpha))(k)=\beta(k)
\end{align*}
which implies that $k \in Per_p(\beta)$.
\end{proof}

For each $p \in \natnump$, consider the action of the symmetric group $Sym(\alphabet^p)$ on $K(\shiftspace)$ defined by
\[\phi \cdot K \mapsto \widehat{\phi}[K]\]
where $\widehat{\phi}$ is the homeomorphism of $\shiftspace$ given by
\[\widehat{\phi}(\gamma)[kp,(k+1)p)=\phi(\gamma[kp,(k+1)p)) \text{ for all } k \in \mathbb{Z} \text{ and for all }\gamma \in \shiftspace.\]
It is not difficult to check that this action is Borel; and that the orbit equivalence relation $\mathtt{D}_p$ of this action is a finite Borel equivalence relation. We are now ready to present the main theorem of this section.

\begin{theorem}\label{theorem-e1reduction} The topological conjugacy relation on $\spaceoftoeplitzsubshiftgr$ is Borel reducible to $E_1$.
\end{theorem}
\begin{proof} Let $\backsim_{\alphabet}$ be the equivalence relation on $(\natnump)^{\mathbb{N}} \times (\Fin(K(\shiftspace)))^{\mathbb{N}}$ defined by
\[(r,(F_i)_{i \in \mathbb{N}}) \backsim_{\alphabet} (s,(F'_i)_{i \in \mathbb{N}}) \Longleftrightarrow r=s\ \wedge\ \exists j\ \forall i \geq j\ (F_i,F'_i) \in \mathtt{D}_{r_i}^{\fin}\]

It is easily checked that each $\mathtt{D}_{p}^{\fin}$ is a finite Borel equivalence relation and hence is smooth. It follows that $\backsim_{\alphabet}$ is Borel reducible to $E_1$. Thus it is sufficient to prove that the topological conjugacy relation on $\spaceoftoeplitzsubshiftgr$ is Borel reducible to $\backsim_{\alphabet}$.
Let $f: \spaceoftoeplitzsubshiftgr \rightarrow (\natnump)^{\mathbb{N}} \times (\Fin(K(\shiftspace)))^{\mathbb{N}}$ be the map given by
\[f(O)=(\tau(O), \chi(O))\]
where $\tau(O)$ is the natural factorization of the scale of $O$ defined in \S 2.6 and
\[\chi(O)_i=\{\seqshift^{\lfloor j/2 \rfloor}[W]: W \in Parts_*(O,\tau(O)_i)\ \wedge\ length(W)=j\}\]
for all $i \in \mathbb{N}$. In other words, $\chi(O)_i$ is the subset of $Parts(O,\tau(O)_i)$ consisting of those elements which position the midpoints of the filled $\tau(O)_i$-blocks in the $\tau(O)_i$-skeleton of $O$ at index $0$. (If such a block has even length, then its ``midpoint" is defined to be the index which cuts the block in such a way that there is one more non-blank symbol on its left than on its right.)

We claim that $f$ is a Borel reduction from the topological conjugacy relation on $\spaceoftoeplitzsubshiftgr$ to the equivalence relation $\backsim_{\alphabet}$. It is straightforward to check that $f$ is Borel and we will skip the tedious details.

To see that $f$ is a reduction, pick $O, O' \in \spaceoftoeplitzsubshiftgr$ such that $O$ and $O'$ are topologically conjugate and let $\pi: O \rightarrow O'$ be a topological conjugacy. Recall that topologically conjugate Toeplitz subshifts have the same scale and hence $\tau(O)=\tau(O')$. Let $(r_i)_{i \in \mathbb{N}}$ be the sequence $\tau(O)$. Since $O$ and $O'$ both have growing blocks with respect to $(r_i)_{i \in \mathbb{N}}$, there exists $n_0$ such that the minimal lengths of the filled $r_i$-blocks of $O$ and $O'$ are both greater than $4|\pi|+6$ for all $i \geq n_0$. We claim that
\[(\chi(O)_i,\chi(O')_i) \in \mathtt{D}_{r_i}^{\fin}\]
for all $i \geq n_0$, which implies that $f(O) \backsim_{\alphabet} f(O')$. Let $i \in \mathbb{N}$ be such that $i \geq n_0$. We want to show that
\[ \{[W]_{\mathtt{D}_{r_i}}: W \in \chi(O)_i\}=\{[W]_{\mathtt{D}_{r_i}}: W \in \chi(O')_i\}\]
Pick $W \in \chi(O)_i$. By the definition of $\chi(O)$, $W$ is of the form $\seqshift^{\lfloor n/2 \rfloor}[Z]$ for some set $Z \in Parts_*(O,r_i)$ with $length(Z)=n$. Choose $\alpha \in W$ and set $\beta=\pi(\alpha)$. Let $k=\lfloor n/2 \rfloor -|\pi|-2$ and $k'=\lfloor n/2 \rfloor + |\pi| + 2$.

By the choice of $i$, we have that $n \geq 4|\pi|+6$ and hence $k \geq |\pi|+1$. Since
\[ [-k-|\pi|-1,k+|\pi|+1] \subseteq Per_{r_i}(\alpha)\]
it follows from Proposition \ref{proposition-phole} that $[-k,k] \subseteq Per_{r_i}(\beta)$ and hence the subblock $\beta[-k,k]$ is a part of some filled $r_i$-block of $Skel(\beta,r_i)$. Similarly, it follows from Proposition \ref{proposition-phole} that there are at least two $r_i$-holes in $Skel(\beta,r_i)$ along the interval $[-k',k']$ since $Skel(\alpha,r_i)$ has two $r_i$-holes at the indices $-1 - \lfloor n/2 \rfloor$ and $n - \lfloor n/2 \rfloor$. Let $q' < 0 < q$ be the $r_i$-holes in the skeleton $Skel(\beta,r_i)$ such that \[Skel(\beta,r_i)(q'')\neq \blanksymbol\]
for all $q' < q'' < q$. Clearly we have that $-k' \leq q' < -k < k < q \leq k'$. Set
\[j= \lceil (q+q')/2 \rceil\]
Notice that the filled $r_i$-block to which $\beta[-k,k]$ belongs is $\beta[q'+1,q-1]$ and the midpoint of this filled $r_i$-block is $j$. Hence $\seqshift^{j}[\pi[W]] \in \chi(O')_i$.

By the choice of $i$, we know that the minimal lengths of the filled-$r_i$-blocks of $\alpha$ and $\seqshift^j(\beta)$ are both greater than $4|\pi|+6$. Since $W$ and $\seqshift^j[\pi[W]]$ both position the midpoints of the corresponding filled $r_i$-blocks at $0$, we have that
\[ [-2|\pi|-2,2|\pi|+2] \subseteq Per_{r_i}(\alpha), Per_{r_i}(\seqshift^j(\beta)) \]
On the other hand, it follows from the previous inequalities that
\[ j= \lceil (q+q')/2 \rceil \leq \lceil (k'-k)/2 \rceil \leq |\pi|+2 \]
and hence the topological conjugacy $\seqshift^j \circ \pi$ and its inverse can be given by some block codes of length at most $2|\pi|+2$. Consequently, Lemma \ref{lemma-obscurelemma} implies that there exists $\phi \in Sym(\alphabet^{r_i})$ such that
\[ \phi(\alpha[l r_i,(l+1)r_i))=(\seqshift^j(\beta))[lr_i,(l+1)r_i)\]
for all $l \in \mathbb{Z}$. Then it easily follows from the proof of Lemma \ref{lemma-notobscurelemma} that the induced homeomorphism $\widehat{\phi}$ bijectively maps $W$ onto $\seqshift^j[\pi[W]]$. Therefore, $W$ and $\seqshift^j[\pi[W]]$ are $\mathtt{D}_{r_i}$-equivalent which shows that
\[ \{[W]_{\mathtt{D}_{r_i}}: W \in \chi(O)_i\} \subseteq \{[W]_{\mathtt{D}_{r_i}}: W \in \chi(O')_i\}\]
Carrying out this argument symmetrically, we easily obtain $(\chi(O)_i,\chi(O')_i) \in \mathtt{D}_{r_i}^{\fin}$. Hence, $f(O) \backsim_{\alphabet} f(O')$ whenever $O$ and $O'$ are topologically conjugate.

Now pick $O, O' \in \spaceoftoeplitzsubshiftgr$ and assume that $f(O) \backsim_{\alphabet} f(O')$. Then $\tau(O)=\tau(O')$; and for some sufficiently large $i$ and some $W \in Parts(O,\tau(O)_i)$ is bijectively mapped onto some $Z \in Parts(O',\tau(O)_i)$ via a homeomorphism $\widehat{\phi}$ induced by a permutation $\phi \in Sym(\alphabet^{\tau(O)_i})$ . In this case, $\widehat{\phi}$ can be extended to a topological conjugacy between $O$ and $O'$ as in the proof of Lemma \ref{lemma-notobscurelemma}. Therefore, $O$ and $O'$ are topologically conjugate.
\end{proof}
\begin{proof}[Proof of Theorem \ref{theorem-mainresulttoeplitz}] It follows from Theorem \ref{theorem-e1reduction} that the topological conjugacy relation on $\spaceoftoeplitzsubshiftgr$ is hypersmooth and hence is hyperfinite by Corollary \ref{corollary-maincorollary}. Consequently, its restrictions onto Borel subsets of $\spaceoftoeplitzsubshiftgr$ are hyperfinite. In particular, the topological conjugacy relation on $\spaceoftoeplitzsubshiftsep$ is hyperfinite.
\end{proof}

\section{Concluding Remarks}
It is easily seen from the proofs of Lemma \ref{lemma-obscurelemma} and Lemma \ref{lemma-notobscurelemma} that if $O$ and $O'$ are topologically conjugate Toeplitz subshifts, then some elements of the partition $Parts(O,p)$ are mapped onto some elements of the partition $Parts(O',p)$ under the natural action of $Sym(n^p)$ on $K(\shiftspace)$ for a sufficiently large factor $p$ of the common scale. The proof of Theorem \ref{theorem-e1reduction} relies on the fact that we can eventually identify the ``correct" subset of each partition $Parts(O,p)$ in a Borel way for Toeplitz subshifts with growing blocks. It is natural to attempt to find such a Borel choice for arbitrary Toeplitz subshifts. However, it is not clear to us how to identify the ``correct" subset of each $Parts(O,p)$ in a Borel way without imposing any conditions on the $p$-skeleton structures.

We should also note that it essentially follows from Lemma \ref{lemma-obscurelemma} and Lemma \ref{lemma-notobscurelemma} that topological conjugacy of pointed Toeplitz subshifts \textit{with Toeplitz points} is hyperfinite. Consequently, one might expect to prove the hyperfiniteness of the topological conjugacy relation on Toeplitz subshifts by reducing topological conjugacy to pointed topological conjugacy.
\begin{question} Does there exist a Borel map $f: \spaceoftoeplitzsubshift \rightarrow \shiftspace$ such that for all $O,O' \in \spaceoftoeplitzsubshift$,
\begin{itemize}
\item[a.] $f(O) \in O$,
\item[b.] $f(O)$ is a Toeplitz sequence, and
\item[c.] $(O,f(O))$ is pointed topologically conjugate to $(O',f(O'))$ whenever $O$ and $O'$ are topologically conjugate?
\end{itemize}
\end{question}
We suspect that there does not exist such a Borel map. However, we should note that such a Borel map exist may exist for various classes of Toeplitz subshifts. For example, consider the class of Toeplitz subshifts with single holes whose scales do not contain an even factor. For each such Toeplitz subshift $O$, we can construct the sequence $(W_i)_{i \in \mathbb{N}}$ of closed sets, where $W_i \in Parts(O,r_i)$, $r_i$ is a factorization of the scale of $O$ with respect to which $O$ has single holes, and the $r_i$-skeleton of $W_i$ positions the single hole in the interval $[0,r_i)$ at index $(r_i-1)/2$. Then an argument similar to the proof of Theorem \ref{theorem-e1reduction} shows that the Borel map $O \mapsto \bigcap_{i \in \mathbb{N}} W_i$ satisfies the requirements.

\textbf{Acknowledgements} This work is a part of the author's PhD thesis under the supervision of Simon Thomas. The author is grateful to Simon Thomas for his invaluable guidance and many fruitful discussions. This research was partially supported by Simon Thomas through the NSF grant DMS-1101597.
\bibliography{referencetoeplitz}{}
\bibliographystyle{amsalpha}
\end{document}